\documentclass[11pt]{amsart}
\usepackage{amsthm,amsmath,amssymb,enumerate}
\title[]{Asymptotic stability for K\"ahler-Ricci solitons}
\author[]{Ryosuke Takahashi}

\address{Graduate School of Mathematics, Nagoya University, Furo-cho, Chikusa-ku, Nagoya, 464-8602, Japan}

\email{m11036a@math.nagoya-u.ac.jp}


\thanks{This is the author's accepted version. The final publication is available at Springer via http://dx.doi.org/10.1007/s00209-015-1518-4.}

\makeatletter

\@addtoreset{equation}{section}
\makeatother

{
\theoremstyle{definition}
\newtheorem{definition}{Definition}[section]
\newtheorem*{acknowledgements}{Acknowledgements}
}

{
\theoremstyle{plain}
\newtheorem{theorem}{Theorem}[section]
\newtheorem{proposition}{Proposition}[section]
\newtheorem{lemma}{Lemma}[section]
\newtheorem{corollary}{Corollary}[section]
}

{
\theoremstyle{remark}
\newtheorem{remark}{Remark}[section]

}
\begin{document}
\begin{abstract}
Let $X$ be a Fano manifold. We say that a hermitian metric $\phi$ on $-K_X$ with positive curvature $\omega_{\phi}$ is a K\"ahler-Ricci soliton if it satisfies the equation ${\rm Ric}(\omega_{\phi}) - \omega_{\phi} = L_{V_{KS}} \omega_{\phi}$ for some holomorphic vector field $V_{KS}$. The candidate for a vector field $V_{KS}$ is uniquely determined by the holomorphic structure of $X$ up to conjugacy, hence depends only on the holomorphic structure of $X$. We introduce a sequence $\{V_k\}$ of holomorphic vector fields which approximates $V_{KS}$ and fits to the quantized settings. Moreover, we also discuss about the existence and convergence of the quantized K\"ahler-Ricci solitons attached to the sequence $\{V_k\}$.
\end{abstract}
\maketitle
\tableofcontents
\section{Introduction} \label{sect:1}
Let $X$ be an $n$-dimensional Fano manifold and $PSH(X,-K_X)$ the set of (possibly singular) hermitian metrics $\phi$ on the anti-canonical bundle $-K_X$ with positive curvature current
\[
\omega_{\phi}:=\frac{\sqrt{-1}}{2 \pi} \partial \bar{\partial} \phi.
\]
We regard $\phi$ as a psh weight, i.e., a psh function on $K_X \backslash \{\text{0-section}\}$ satisfying the log-homogeneity property (e.g., see \cite{BB10} for more detail). Let ${\mathcal H} (X, -K_X)$ be the subset of $PSH(X,-K_X)$ consisting of all smooth psh weights. We say that a metric $\phi \in {\mathcal H}(X,-K_X)$ is a K\"ahler-Ricci soliton if it satisfies the equation
\[
{\rm Ric}(\omega_{\phi}) - \omega_{\phi} = L_{V_{KS}} \omega_{\phi}
\]
for some holomorphic vector field $V_{KS}$ (\textbf{K\"ahler-Ricci soliton vector field}), where $L_{V_{KS}}$ denotes the Lie derivative with respect to $V_{KS}$. K\"ahler-Ricci solitons arise from the geometric analysis, such as K\"ahler-Ricci flow, and have been studied extensively for recent years. For instance, Chen-Wang \cite{CW14} solved the Hamilton-Tian conjecture, which says that K\"ahler-Ricci flow on a Fano manifold always converges to a singular K\"ahler-Ricci soliton defined on a singular normal Fano variety in the sense of Cheeger-Gromov.

In this paper, we study the existence problem of K\"ahler-Ricci solitons. As shown in \cite{TZ02}, a necessary condition for the existence of a K\"ahler-Ricci soliton is the vanishing of \textbf{the modified Futaki invariant}:
let ${\rm Aut}_0(X)$ be the identity component of the automorphism group of $X$. Since ${\rm Aut}_0(X)$ is a linear algebraic group \cite{Fuj78}, we obtain a semidirect decomposition
\[
{\rm Aut}_0(X) = {\rm Aut}_r (X) \ltimes R_u,
\]
where ${\rm Aut}_r (X)$ is a reductive algebraic subgroup (uniquely determined up to conjugacy of ${\rm Aut}_0(X)$), which is the complexfication of a maximal compact subgroup $K$, and $R_u$ is the unipotent radical. We often identify a holomorphic vector field $V$ such that ${\rm Im}(V) \in {\mathfrak k}:={\rm Lie}(K)$ with its imaginary part $\xi_V:={\rm Im}(V) \in {\mathfrak k}$. Let $T_c$ be a real torus defined as the identity component of the center of $K$ and put ${\mathfrak t}_c:={\rm Lie}(T_c)$. Then Tian-Zhu (cf. \cite{TZ00}, \cite{TZ02}) showed that all K\"ahler-Ricci solitons are contained in the space of $K$-invariant smooth psh weights ${\mathcal H}(X,-K_X)^K$ and $\xi_{V_{KS}}$ is contained in ${\mathfrak t}_c$, which is uniquely determined as the minimizer of the following proper strictly convex function on ${\mathfrak k}$:
\[
{\mathcal F}(V):=\int_X e^{m_{\phi} (\xi_V)} MA(\phi),
\]
where $\phi \in {\mathcal H}(X,-K_X)^K$, $MA(\phi):=\frac{\omega_{\phi}^n}{c_1(X)^n}$ and $m_{\phi}$ is the moment map with respect to $\phi$. The function ${\mathcal F}$ is a holomorphic invariant, i.e., independent of a choice of $\phi \in {\mathcal H}(X,-K_X)^K$ and its derivative at $V$:
\[
{\rm Fut}_V (W):= - \int_X m_{\phi}(\xi_W) e^{m_{\phi} (\xi_V)} MA(\phi)
\]
is called the modified Futaki invariant. 

Thereafter, Berman-Nystr\"om \cite{BN14} generalized the modified Futaki invariant for singular normal Fano varieties with a torus action, and introduced the notion of algebro-geometric stability, called K-polystability. They also showed that any Fano manifolds admitting a K\"ahler-Ricci soliton are K-polystable. The converse is also true at least in the case of K\"ahler-Einstein metrics (cf. \cite{CDS15}, \cite{Tia15}), however, it is still open in general.

K-polystability is not only the notion of stability for the exsitence problem of K\"ahler-Ricci solitons. We study several functionals on the space of hermitian metrics and their asymptotics near the boundary. Then the notion of analytical K-polysitability is defined as the strong properness (or coercivity) of the modified Ding functional. On the other hand, we also consider its analogue in the space ${\mathcal H}_k$ of hermitian inner products on $H^0(X,-kK_X)$ and study their asymptotics as $k$ tends to infinity. Critical points of the quantized functionals are called balanced metrics. The existence of a balanced metric is closely related to the stability of the projective embedding of $X$ (cf. \cite{Don02}, \cite{Don09}). Berman-Nystr\"om \cite{BN14} showed that there exist a certain kind of balanced metrics, called \textbf{quantized K\"ahler-Ricci solitons} and this sequence of metrics converges to a K\"ahler-Ricci soliton as $k$ tends to infinity under some strong assumptions.

The main purpose of this paper is to perturb the sequence of quantized K\"ahler-Ricci solitons and show their existence and convergence under weaker assumptions. First, we construct a sequence $\{V_k\}$ of holomorphic vector fields which approximates $V_{KS}$ and fits to the quantized settings. More concretely, this sequence is given as the following:
\begin{theorem} \label{theo:1.1}
Let $X$ be a Fano manifold and $K$ be a maximal compact subgroup of ${\rm Aut}_r (X)$. Then for sufficiently large $k$, there exists a holomorphic vector field $V_k$ such that its imaginary part is contained in ${\mathfrak t}_c$ and the corresponding quantized modified Futaki invariant at level $k$ vanishes on ${\mathfrak k}^{\mathbb C}$. The vector field $V_k$ is characterized as the unique minimizer of the quantization of the function ${\mathcal F}|_{{\mathfrak t}_c}$ at level $k$ and  converges to $V_{KS}$ as $k \rightarrow \infty$ in the usual topology of the finite dimensional vector space ${\mathfrak t}_c$. 
\end{theorem}
Second, we introduce the quantized K\"ahler-Ricci solitons attached to this sequence and show that:
\begin{theorem} \label{theo:1.2}
Assume that $(X, V_{KS})$ is strongly analytically K-polystable (i.e., the corresponding modified Ding functional is coercive modulo ${\rm Aut}_0(X, V_{KS})$), then there exists a quantized K\"ahler-Ricci soliton attached to $V_k$ if $k$ is sufficiently large, which is unique modulo the action of ${\rm Aut}_0(X, V_{KS})$ and as $k \rightarrow \infty$, the corresponding Bergman metrics on $X$ converge weakly, modulo automorphisms, to a K\"ahler-Ricci soliton on $(X,V_{KS})$.
\end{theorem}
As a corollary, we have the following:
\begin{corollary} \label{coro:1.1}
Assume that $X$ is strongly analytically K-polystable (i.e., the Ding functional is coercive modulo ${\rm Aut}_0 (X)$), then there exists a quantized K\"ahler-Einstein metric attached to $V_k \rightarrow 0$ if $k$ is sufficiently large, which is unique modulo the action of ${\rm Aut}_0 (X)$ and as $k \rightarrow \infty$, the corresponding Bergman metrics on $X$ converge weakly, modulo automorphisms, to a K\"ahler-Einstein metric on $X$.
\end{corollary}
The crucial point is that in our results, we need not to assume that the vanishing of all the higher order (modified) Futaki invariants, which is, in the case of $V_{KS} \equiv 0$, an obstruction to the asymptotic Chow semi-stability (cf. \cite{Fut04}). Thus we can apply our results to even asymptotically Chow unstable Fano manifolds like \cite{OSY12}.

Now we describe the content of this paper.
In the next section, we review the basic of several functionals on the space of smooth psh weights ${\mathcal H}(X, -K_X)$. The standard reference for this section is \cite{BN14}. However, in order to prove the convergence of quantized metrics in Theorem \ref{theo:1.2}, we need to extend functionals for (possibly singular) psh weights, which requires a lot of knowledge about the pluripotential theory (e.g. \cite{BBGZ13}, \cite{BE10}, \cite{BN14}). So the reader should see these references as needed.

In Section \ref{sect:3}, we introduce the quantization of the modified Futaki invariant ${\rm Fut}_{V, k}$ and show that the functional ${\rm Fut}_{V, k}$ restricted to ${\mathfrak t}_c$ is strictly proper convex, and hence has a unique minimizer $V_k$. This observation and the quantization formula (cf. Lemma \ref{lemm:3.2}) yield Theorem \ref{theo:1.1}. Then we review the basic properties of the quantized functionals on ${\mathcal H}_k$ studied in \cite{BN14} and prove Theorem \ref{theo:1.2}. The heart of the proof of Theorem \ref{theo:1.2} consists of mainly two ideas:\\
(1) While Berman-Nystr\"om considered the torus $T_{KS}$ generated by the K\"ahler-Ricci soliton vector field $V_{KS}$, we consider the identity component of the center $T_c (\supset T_{KS})$ and the space of $T_c$-invariant hermitian metrics ${\mathcal H}(X, -K_X)^{T_c}$. Actually, this setting seems to be natural since all of $\xi_{V_k}$ lie in its Lie algebra ${\mathfrak t}_c$ by Theorem \ref{theo:1.1}.\\
(2) The condition ${\rm Fut}_{V_{KS}} \equiv 0$ (resp. ${\rm Fut}_{V_k, k} \equiv 0$) leads to the ${\rm Aut}_0 (X, V_{KS})$-invariance of the functional ${\mathcal D}_{g_{V_{KS}}}$ (resp. ${\mathcal D}_{g_{V_k}} ^{(k)}$). Hence the problem can be reduced to estimate the difference ${\mathcal D}_{g_{V_{KS}}}^{(k)} - {\mathcal D}_{g_{V_k}} ^{(k)}$, which is linear along geodesics. On the other hand, the standard exhaustion function $J^{(k)}$ has at least linear growth along geodesics. Therefore the absolute of ${\mathcal D}_{g_{V_{KS}}}^{(k)} - {\mathcal D}_{g_{V_k}} ^{(k)}$ is bounded above by an affine function $\epsilon_k J^{(k)} + \epsilon_k'$ of $J^{(k)}$ with some positive numbers $\epsilon_k \rightarrow 0$ and $\epsilon_k' \rightarrow 0$. This leads to the coercivity of ${\mathcal D}_{g_{V_k}} ^{(k)}$ and therefore the existence of the quantized K\"ahler-Ricci soliton attached to $V_k$.

Finally, we mention a result for extremal K\"ahler metrics proved by Mabuchi \cite{Mab09}, which says that any polarized manifolds admitting an extremal K\"ahler metric are asymptotically Chow stable relative to an algebraic torus. It seems that such a stability for K\"ahler-Ricci solitons has never been discussed, but it is known that relative Chow stability leads to the existence of ``polybalanced metrics'' (cf. \cite{Mab11}). That is why Theorem \ref{theo:1.2} can be seen as an analogue of Mabuchi's result indirectly.

\begin{acknowledgements}
The author would like to express his gratitude to Professor Ryoichi Kobayashi for his advice on this article, and to the referee for useful suggestions that helped him to improve the original manuscript. The author is supported by Grant-in-Aid for JSPS Fellows Number 25-3077.
\end{acknowledgements}
\section{Functionals on ${\mathcal H}(X,-K_X)^T$ and analytic K-polystability} \label{sect:2}
Let $X$ and $V_{KS}$ be as in Section \ref{sect:1}. Then we see that $V_{KS}$ generates a torus action on $(X, -K_X)$:
\begin{lemma}[\cite{BN14}, Lemma 2.13] \label{lemm:2.1}
There is a real torus $T_{KS} \subset {\rm Aut}_r (X)$ which acts on $(X, -K_X)$ such that the imaginary part of $V_{KS}$ can be identified with an element $\xi_{V_{KS}} \in {\mathfrak t}_{KS}:={\rm Lie}(T_{KS})$ and ${\mathcal H}(X,-K_X)^{V_{KS}} = {\mathcal H}(X,-K_X)^{T_{KS}}$.
\end{lemma}
Let $T \subset {\rm Aut}_r (X)$ be an $m$-dimensional real torus acting on $(X, -K_X)$.
For $\phi \in {\mathcal H}(X,-K_X)^T$, we denote the moment map with respect to $\phi$ by
\[
m_{\phi} \colon X \rightarrow m_{\phi}(X)=:P \subset {\mathfrak t}^* \simeq {\mathbb R}^m,
\]
where we identify ${\mathfrak t}^* \simeq {\mathbb R}^m$ using a inner product on ${\mathfrak t}$.
The image $P$ is a compact convex polytope of dimension $m$, characterized as the support of the Duistermaat-Heckman measure
\[
\nu^T:=(m_{\phi})_{\ast} MA(\phi),
\]
which is independent of a choice of $\phi \in {\mathcal H}(X,-K_X)^T$ (cf. Proposition \ref{prop:3.1}). 

Now we recall several functionals on ${\mathcal H}(X,-K_X)^T$ which play a central role in the study of K\"ahler-Ricci solitons. Let $\phi_0 \in {\mathcal H}(X,-K_X)^T$ be a reference metric and $g$ a positive continuous function on the moment polytope $P$. We normalize $g$ so that $g \nu^T$ is a probability measure on $P$. Following \cite[Section 2.4 and 2.6]{BN14}, we define the $g$-Monge-Amp\`ere energy by the formula
\[
d {{\mathcal E}_g}|_{\phi} (\dot{\phi})= \int_X \dot{\phi} MA_g (\phi), \;\; {\mathcal E}_g (\phi_0) = 0,
\]
where $MA_g (\phi) := g(m_{\phi}) MA(\phi)$ denotes the $g$-Monge-Amp\`ere measure.
Then the functional ${\mathcal E}_g$ satisfies the scaliing property:
\[
{\mathcal E}_g (\phi + c) = {\mathcal E}_g (\phi) + c
\]
for all $\phi \in {\mathcal H}(X, -K_X)^T$ and $c \in {\mathbb R}$.
Let $\mu_{\phi}$ be a measure on $X$ given by the local expression
\[
\mu_{\phi} = e^{-\phi_U} (\sqrt{-1})^n dz_1 \wedge d\bar{z}_1 \wedge \cdots \wedge dz_n \wedge d\bar{z}_n,
\]
where $(U; z_1, \ldots, z_n)$ denotes a holomorphic local coordinates and $\phi_U:= \log \left| \frac{\partial}{\partial z_1} \wedge \cdots  \wedge \frac{\partial}{\partial z_n} \right|_{\phi}^2$.
We define functionals $J_g$ and ${\mathcal D}_g$ by
\[
J_g (\phi) = - {\mathcal E}_g(\phi) + {\mathcal L}_{\mu_0} (\phi), \;\; {\mathcal L}_{\mu_0} (\phi):= \int_X (\phi-\phi_0)d\mu_0,
\]
\[
{\mathcal D}_g (\phi) = - {\mathcal E}_g (\phi) + {\mathcal L} (\phi), \;\; {\mathcal L} (\phi):=-\log \int_X d\mu_{\phi},
\]
where $\mu_0:=MA(\phi_0)$ is a fixed probability measure on $X$. 
One can easily see that $J_g$ and ${\mathcal D}_g$ are invariant under scaling of metrics, i.e., functionals on ${\mathcal H}(X, -K_X)^T/ {\mathbb R}$. When $g \equiv 1$, we simply write these functionals by ${\mathcal E}$, $J$ and ${\mathcal D}$ respectively.
\begin{remark} \label{rema:2.1}
We can extend these functionals to the space ${\mathcal E}^1(X,-K_X)^T$ of all $T$-invariant (possibly singular) psh weights with finite energy  (cf. \cite[Section 2 and Section 3]{BN14}).
Moreover, the functional $J$ defines an exhaustion function on ${\mathcal E}^1(X,-K_X)/ {\mathbb R}$ in the sense that each level set of $J$ is compact in $L^1$-topology (cf. \cite[Lemma 3.3]{BBGZ13}).
\end{remark}
Next we set $T=T_{KS}$ and $g=g_{V_{KS}}:= \exp (\langle \xi_{V_{KS}}, \cdot \; \rangle)$. Then the corresponding functional ${\mathcal D}_g={\mathcal D}_{g_{V_{KS}}}$ is called the modified Ding functional.
By \cite[Lemma 3.4]{BN14}, we have
\[
\frac{d}{dt} {\mathcal D}_{g_{V_{KS}} }(\exp(tW)^* \phi) = {\rm Fut}_{V_{KS}} (W).
\]
Moreover, critical points of ${\mathcal D}_{g_{V_{KS}}}$ are K\"ahler-Ricci solitons with respect to $V_{KS}$.
\begin{definition}[\cite{BN14}, Section 3.6] \label{defi:2.1}
We say that a pair $(X, V_{KS})$ is strongly analytically K-polystable if the modified Ding functional ${\mathcal D}_{g_{V_{KS}}}$ is coercive modulo ${\rm Aut}_0 (X,V_{KS})$, i.e.,
\[
{\mathcal D}_{g_{V_{KS}}} (\phi) \geq \delta \underset{F \in {\rm Aut}_0 (X, V_{KS})}{\rm inf} J (F^*\phi) -C, \;\; \phi \in {\mathcal H}(X,-K_X)^{T_{KS}}
\]
for some positive constants $\delta$ and $C$, where ${\rm Aut}_0 (X, V_{KS})$ be a subgroup of ${\rm Aut}_0 (X)$ consisting of elements which commute with the action generated by $V_{KS}$.
\end{definition}
\begin{theorem}[\cite{BN14}, Theorem 3.11] \label{theo:2.1}
If a pair $(X, V_{KS})$ is strongly analytically K-polystable, then $(X, V_{KS})$ admits a K\"ahler-Ricci soliton (as a unique minimizer of the modified Ding functional up to the action of ${\rm Aut}_0 (X, V_{KS})$).
\end{theorem}
\section{The quantized setting} \label{sect:3}
\subsection{The quantization of ${\mathcal F}$-functional, the modified Futaki invariant and the Duistermaat-Heckman measure} \label{sect:3.1}
Let $X$ and $T$ be as in Section \ref{sect:2}.
\begin{lemma} \label{lemm:3.1}
Assume that $m:={\rm dim}T$ is greater than $1$ . Then the polytope $P$ contains the origin in its interior ${\rm int}(P)$.
\end{lemma}
\begin{proof}
Since the lift to $-K_X$ is canonical, we have an equation
\[
-\Delta_{\partial} m_{\phi} (\xi_V) + m_{\phi} (\xi_V) + V(\kappa_{\phi}) = 0
\]
for all $\xi_V \in {\mathfrak t}$, where $\kappa_{\phi}$ is the function defined by 
\[
{\rm Ric}(\omega_{\phi}) - \omega_{\phi} = \frac{\sqrt{-1}}{2 \pi} \partial \bar{\partial} \kappa_{\phi}.
\]
Integrating by parts, we find that
\begin{equation}
\int_X m_{\phi} (\xi_V) e^{\kappa_{\phi}} MA(\phi) = 0. \label{eq:3.1}
\end{equation}
Since $m \geq 1$, $m_{\phi}$ is not a constant. Thus the equation \eqref{eq:3.1} implies that  for any $\xi_V$, an inequality $m_{\phi}(\xi_V)>0$ holds on some nonempty open subset of $X$. Now we assume that $0 \not\in {\rm int}(P)$, then we can choose an element $\xi \in {\mathbb R}^m$ so that $H_{\xi} \cap {\rm int}(P) = \emptyset$, where $H_{\xi}$ a hyperplane which is orthogonal to $\xi$ and contains the origin. Then either of $m_{\phi} (\pm \xi_V)$ is semi-negative on $X$. This is a contradiction. 
\end{proof}
We define the functions ${\mathcal F}_k$ and ${\rm Fut}_{V,k}$ as
\begin{eqnarray}
{\mathcal F}_k (W):= k {\rm Trace} (e^{W/k})|_{H^0(X,-kK_X)}, \\
{\rm Fut}_{V,k} (W) := - \left. \frac{d}{dt} \right|_{t=0} {\mathcal F}_k (V+tW).
\end{eqnarray}
We set
\[
N_k:= {\rm dim} H^0(X,-kK_X),
\]
then these functions give the quantization of ${\mathcal F}$ and ${\rm Fut}_V$:
\begin{lemma}[\cite{BN14}, Proposition 4.7 or \cite{Tak14}, Proposition 2.8] \label{lemm:3.2}
Let $V$ be a holomorphic vector field generating a torus action and $W$ a holomorphic vector field generating a ${\mathbb C}^*$-action and commuting with $V$. Then we have identities
\[
{\mathcal F}(V) = \lim_{k \rightarrow \infty} \frac{{\mathcal F}_k (V)}{k N_k}, \label{eq:3.2}
\]
\[
{\rm Fut}_V (W) = \lim_{k \rightarrow \infty} \frac{1}{kN_k} {\rm Fut}_{V,k}(W). \label{eq:3.3}
\]
\end{lemma}
If we apply the equivariant Riemann-Roch formula to ${\rm Fut}_{V, k} (W)$, we have an expansion
\[
{\rm Fut}_{V,k} (W)={\rm Fut}_V ^{(0)} (W) k^{n+1} + {\rm Fut}_V ^{(1)} (W) k^n + \cdots,
\]
where ${\rm Fut}_V ^{(i)} (W)$ is the $i$-th order modified Futaki invariant introduced in \cite[Section 4.4]{BN14}.
\begin{lemma} \label{lemm:3.3}
The function ${\mathcal F}_k |_{{\mathfrak t}_c}$ is a proper strictly convex function if $k$ is sufficiently large.
\end{lemma}
\begin{proof}
We use the following proposition:
\begin{proposition}[\cite{BN14}, Proposition 4.1] \label{prop:3.1}
Let $P_k:=\{ \lambda_i^{(k)} \} \subset {\mathbb Z}^m$ be the set of all weights for the action of the complexified torus $T^{\mathbb C}$ on $H^0(X,-kK_X)$, i.e., there is a decomposition
\[
H^0(X,-kK_X) = \bigoplus_{\lambda_i^{(k)} \in P_k} E_{\lambda_i^{(k)}}.
\]
Then the spectral measure:
\[
\nu_k:=\frac{1}{N_k} \sum_{i=1}^{N_k} \delta_{\lambda_i^{(k)}/k}
\]
supported on $P_k / k$ converges to the Duistermaat-Heckman measure $\nu^T$ weakly as $k \rightarrow \infty$, where $\delta_{\lambda_i^{(k)}/k}$ denotes the Dirac measure at $\lambda_i^{(k)}/k$. In particular, $\nu^T$ does not depend on a choice of $\phi \in {\mathcal H}(X, -K_X)^T$.
\end{proposition}
For any $\xi_V, \xi_W (\neq 0) \in {\mathfrak t}_c$, the functional ${\mathcal F}_k$ along the line $\xi_V + t \xi_W$ ($t \in {\mathbb R}$) can be written as the form
\[
{\mathcal F}_k (V+ tW)= k \sum_{i=1}^{N_k} \exp(v_i^{(k)}/k + t w_i^{(k)}/k),
\]
where $v_i^{(k)}:=\langle \xi_V, \lambda_i^{(k)} \rangle$ and $v_i^{(k)}:=\langle \xi_W, \lambda_i^{(k)} \rangle$ are joint eigenvalues of the commuting action generated by ${\rm Re}(V)$ and ${\rm Re}(W)$. Then Proposition \ref{prop:3.1} implies that the functional ${\mathcal F}_k (V+ tW)$ of $t$ is strictly convex for any $\xi_V, \xi_W \in {\mathfrak t}_c$ if $k$ is sufficiently large and hence ${\mathcal F}_k$ is strictly convex.

In order to prove the properness, let $\{ \xi_{W_j} \} \subset {\mathfrak t}_c \simeq {\mathbb R}^m$ be any sequence such that $|\xi_{W_j}| \rightarrow \infty$ as $j \rightarrow \infty$. 
For $\epsilon>0$, let $P_{\epsilon}$ be the interior compact convex polytope with faces parallel to those of $P$ separated by distance $\epsilon$. By Lemma \ref{lemm:3.1}, we can choose $\epsilon>0$ so that ${\rm int}(P_{\epsilon})$ contains the origin. Then Proposition \ref{prop:3.1} implies that there exists $k_0$ such that for all $k \geq k_0$ and $\xi_W \in {\mathbb R}^m$, there exists an eigenvalue $\lambda_i^{(k)}$ satisfying
\[
\lambda_i^{(k)}/k \in P-P_{\epsilon},
\]
\[
\cos ({\rm angle}(\xi_W, \lambda_i^{(k)})) \geq 1- \epsilon.
\]
For each $\xi_{W_j}$, we choose the eigenvalue $\lambda_{j,i(j)}^{(k)}$ satisfying the above condition. Then we obtain
\[
w_{j,i(j)}^{(k)} := \langle \lambda_{j,i(j)}^{(k)}, \xi_{W_j} \rangle \geq k|\xi_{W_j}| \cdot {\rm inf}_{\xi \in \partial P_{\epsilon}} |\xi| \cdot (1- \epsilon) \rightarrow \infty
\]
as  $j \rightarrow \infty$. Hence we have
\[
{\mathcal F}_k (W_j) = k \sum_{i=1}^{N_k} \exp (w_{j,i}^{(k)}/k) \geq k \exp (w_{j,i(j)}^{(k)}/k) \rightarrow \infty
\]
as $j \rightarrow \infty$.
This completes the proof of Lemma \ref{lemm:3.3}. 
\end{proof}
Let $V_k$ be the unique minimizer of ${\mathcal F}_k |_{{\mathfrak t}_c}$.
\begin{proof}[The proof of Theorem \ref{theo:1.1}]
By Lemmas \ref{lemm:3.2} and \ref{lemm:3.3}, we find that 
the unique minimizer $V_k$ converges to the unique minimizer of ${\mathcal F}$, i.e., the K\"ahler-Ricci soliton vector field as $k \rightarrow \infty$.
Since $V_k \in {\mathfrak t}_c^{\mathbb C}$ and ${\mathcal F}_k$ is adjoint invariant (so is {\rm Trace} in the defining equation \eqref{eq:3.2} of ${\mathcal F}_k$), we have
\[
{\rm Fut}_{V_k, k} ({\rm Ad}_{F} W)= - \left. \frac{d}{dt} \right|_{t=0} {\mathcal F}_k (V_k + t {\rm Ad}_F W) = - \left. \frac{d}{dt} \right|_{t=0} {\mathcal F}_k (V_k + t W) = {\rm Fut}_{V_k, k} (W)
\]
for any $F \in {\rm Aut}_r (X)$ and $W \in {\mathfrak k}^{\mathbb C}$. In particular, set $F_s := \exp(s V)$ ($V \in {\mathfrak k}^{\mathbb C}$) and differentiating at $s=0$ yields
\[
{\rm Fut}_{V_k, k}([V,W])=0.
\]
Moreover, the formula ${\mathfrak k}^{\mathbb C}={\mathfrak t}_c^{\mathbb C} \oplus [{\mathfrak k}^{\mathbb C},{\mathfrak k}^{\mathbb C}]$ yields that ${\rm Fut}_{V_k, k}$ vanishes on the entire space ${\mathfrak k}^{\mathbb C}$. This completes the proof. 
\end{proof}
\subsection{The $g$-Bergman measure and quantized functionals} \label{sect:3.2}
Let $X$ be a Fano manifold, $\mu$ a $T$-invariant measure, and ${\mathcal H}_k$ the space of hermitian inner products on $H^0(X, -kK_X)$. We define two operators:
\[
Hilb_{k, \mu, g} \colon {\mathcal H}(X,-K_X)^T \rightarrow {\mathcal H}_k^T,
\]
\[
FS_k \colon {\mathcal H}_k^T \rightarrow {\mathcal H}(X,-K_X)^T
\]
by the formula:
\begin{equation}
||s_i^{(k)}||_{Hilb_{k, \mu, g} (\phi)}^2 := g(\lambda_i^{(k)}/k)^{-1} \int_X |s_i^{(k)}|^2 e^{-k \phi} d \mu \;\;\; \text{for $s_i^{(k)} \in E_{\lambda_i^{(k)}}$}, \label{eq:3.4}
\end{equation}
\begin{equation}
FS_k (H):= \frac{1}{k} \log \left( \frac{1}{N_k} \sum_{i=1}^{N_k} |s_i|^2 \right), \label{eq:3.5}
\end{equation}
where we denote an $H$-orthonormal basis by $\{s_i\}$ and extend $Hilb_{k, \mu, g}$ to a hermitian inner product on the space $H^0(X, -kK_X)$ by requiring that the different subspace $E_{\lambda_i^{(k)}}$ are orthogonal to each other. We note that the map $FS_k$ is independent of a choice of an $H$-orthonormal basis $\{s_i\}$. Actually, $FS_k (H)$ is just the pull-back of the Fubini-Study metric with respect to $H$. In the case when $\mu=\mu_{\phi}$, we drop the explicit dependence of $\mu$ from the notation and simply write
\[
Hilb_{k, g} (\phi) := Hilb_{k, \mu_\phi, g} (\phi).
\]
Let $H_0:=Hilb_{k, \mu_0} (\phi_0) \in {\mathcal H}_k^{T_c}$ ($\phi_0 \in {\mathcal H}(X,-K_X)^{T_c}$) be a reference metric. We normalize $g$ so that $g \nu_k$ is a probability measure on $P$. We define the quantization of the functional ${\mathcal E}_g$ as
\begin{equation}
{\mathcal E}_g^{(k)} (H) = \sum_{\lambda_i^{(k)} \in P_k} g(\lambda_i^{(k)}/k) {\mathcal E}_{E_{\lambda_i^{(k)}}}^{(k)} (H), \;\; {\mathcal E}_{E_{\lambda_i^{(k)}}}^{(k)} (H) = - \frac{1}{kN_k} \log \det H|_{E_{\lambda_i^{(k)}}}, \label{eq:3.6}
\end{equation}
where we compute the determinant in reference to the metric $H_0$. 
We have an isomorphism
\[
{\mathcal H}_k \simeq GL(N_k, {\mathbb C})/U(N_k)
\]
with respect to $H_0$, which implies that ${\mathcal H}_k$ is a Riemannian symmetric space and therefore geodesics are given as the decomposition of the exponential map and the projection $GL(N_k,{\mathbb C}) \rightarrow GL(N_k,{\mathbb C})/U(N_k)$.
Let $s_i^{(k)} \in E_{\lambda_i^{(k)}}$ be an $H_0$-orthonormal and $H_t$-orthogonal basis. Then any geodesic can be represented by $H_t(s_i^{(k)}, s_i^{(k)}) = e^{-\mu_i^{(k)}t} H_0 (s_i^{(k)}, s_i^{(k)})$ for some $\mu_i^{(k)} \in {\mathbb R}$. Thus we have
\[
\frac{d}{dt} {\mathcal E}_g^{(k)} (H_t) = \frac{1}{kN_k} \sum_{i=1}^{N_k} g(\lambda_i^{(k)}/k) \mu_i^{(k)}.
\]
Hence the functional ${\mathcal E}_g^{(k)}$ has linear growth along geodesics.
We define the quantization of the functionals $J_g$, ${\mathcal D}_g$ as follows:
\begin{equation}
J_g^{(k)} (H) := - {\mathcal E}_g^{(k)}(H) + {\mathcal L}_{\mu_0}(FS_k(H)), \label{eq:3.7}
\end{equation}
\begin{equation}
{\mathcal D}_g^{(k)} (H) := - {\mathcal E}_g^{(k)}(H) + {\mathcal L}(FS_k(H)). \label{eq:3.8}
\end{equation}
These functionals are invariant under scaling of metrics and descend to functionals on the space ${\mathcal H}_k^{T_c}/ {\mathbb R}$. When $g \equiv 1$, we simply write these functionals by ${\mathcal E}^{(k)}$, $J^{(k)}$ and ${\mathcal D}^{(k)}$ respectively.

Now we will explain that the quantized functionals ${\mathcal E}_g^{(k)}$, $J_g^{(k)}$ and ${\mathcal D}_g^{(k)}$ are, to the letter, the quantization of the corresponding functionals on ${\mathcal H}(X, -K_X)^T$ respectively. We start with mentioning the $g$-Bergman function:
\begin{equation}
\rho_{k, \mu_0, g} (\phi):= \sum_{\lambda_i^{(k)} \in P_k} g(\lambda_i^{(k)}/k) \rho_{k, \mu_0} (\phi) \label{eq:3.9}
\end{equation}
and $g$-Bergman measure
\begin{equation}
\beta_{k, \mu_0, g} (\phi):= \frac{1}{N_k} \rho_{k, \mu_0, g} (\phi) \cdot \mu_0, \label{eq:3.10}
\end{equation}
where $\rho_{k, \mu_0, g} (\phi)$ is the ordinary Bergman function of the subspace $E_{\lambda_i^{(k)}}$. We use the following convergence of measures:
\begin{proposition}[\cite{BN14}, Proposition 4.4] \label{prop:3.2}
Assume that $g$ is smooth, then for any $\phi \in {\mathcal H}(X,-K_X)^T$, we have the uniform convergence
\[
\beta_{k, \mu_0, g} (\phi) \rightarrow MA_g (\phi).
\]
\end{proposition}
Now we are ready to prove the quantization formula (cf. \cite[Proposition 4.5]{BN14}).
\begin{proposition} \label{prop:3.3}
The following pointwise convergence holds as $k \rightarrow \infty$:
\[
{\mathcal E}_g^{(k)} (Hilb_{k, \mu_0} (\phi)) \rightarrow {\mathcal E}_g (\phi),
\]
\[
J_g^{(k)} (Hilb_{k, \mu_0} (\phi)) \rightarrow J_g (\phi),
\]
\[
{\mathcal D}_g^{(k)} (Hilb_{k, \mu_0} (\phi)) \rightarrow {\mathcal D}_g (\phi).
\]
\end{proposition}
\begin{proof}
The direct computation yields that
\begin{eqnarray*}
{\mathcal E}_g^{(k)} (Hilb_{k, \mu_0} (\phi)) &=& \int_0^{1} \left( \frac{d}{dt}{\mathcal E}_g^{(k)} (Hilb_{k, \mu_0} (t \phi + (1-t)\phi_0) \right) dt \\
&\hbox{}& (\text{because ${\mathcal E}_g^{(k)} (Hilb_{k, \mu_0} (\phi_0)) = 0$})\\
&=& \int_0^{1} \int_X (\phi- \phi_0) \beta_{k, \mu_0, g} (t \phi + (1-t) \phi_0) \\
&\hbox{}& (\text{by \cite[Proposition 4.5]{BN14}}).
\end{eqnarray*}
Combining with Proposition \ref{prop:3.2}, we have
\[
{\mathcal E}_g^{(k)} (Hilb_{k, \mu_0} (\phi)) \rightarrow \int_0^{1} \int_X (\phi- \phi_0) MA_g (t \phi + (1-t) \phi_0) = {\mathcal E}_g (\phi).
\]
Moerover, by definition,
\begin{equation}
FS_k \circ Hilb_{k, \mu_0, g} (\phi) - \phi = \frac{1}{k} \log \left( \frac{1}{N_k} \rho_{k, \mu_0, g} (\phi) \right). \label{eq:3.11}
\end{equation}
Thus using Proposition \ref{prop:3.2} again, we have a uniform convergence
\[
FS_k \circ Hilb_{k, \mu_0, g} (\phi) \rightarrow \phi.
\]
The last two parts follow from the defining equations \eqref{eq:3.7}, \eqref{eq:3.8} and pointwise convergence ${\mathcal E}_g^{(k)} \circ Hilb_{k, \mu_0} \rightarrow {\mathcal E}_g$. 
\end{proof}
\subsection{Modification of quantized K\"ahler-Ricci solitons} \label{sect:3.3}
We adopt the same notation as in Section \ref{sect:3.2}. We set $T=T_{KS}$.
\begin{definition}[\cite{BN14}, Section 4] \label{defi:3.1}
We say that a metric $H_k \in {\mathcal H}_k^{T_{KS}}$ is a quantized K\"ahler-Ricci soliton if it satisfies the equation
\[
Hilb_{k, g_{V_{KS}}} \circ FS_k (H_k) = H_k.
\]
\end{definition}
Berman-Nystr\"om showed the following:
\begin{theorem}[\cite{BN14}, Theorem 1.7] \label{theo:3.1}
Assume that $(X, V_{KS})$ is strongly analytically K-polystable (i.e., the corresponding modified Ding functional is coercive modulo ${\rm Aut}_0(X, V_{KS})$) and all the higher order modified Futaki invariants of $(X, V_{KS})$  vanish, then there exists a quantized K\"ahler-Ricci soliton if $k$ is sufficiently large, which is unique modulo the action of ${\rm Aut}_0(X, V_{KS})$ and as $k \rightarrow \infty$, the corresponding Bergman metrics on $X$ converge weakly, modulo automorphisms, to a K\"ahler-Ricci soliton on $(X,V_{KS})$.
\end{theorem}

We want to weaken the assumption in the above theorem. For this, we set $T=T_c$ and introduce a slight modification of the notion of quantized K\"ahler-Ricci solitons:
\begin{definition} \label{defi:3.2}
Let $\{V_k\}$ be a sequence of holomorphic vector fields constructed in Section \ref{sect:3.1}.
We say that a metric $H_k \in {\mathcal H}_k^{T_c}$ is a quantized K\"ahler-Ricci soliton attached to $V_k$ if it satisfies the equation
\[
Hilb_{k, g_{V_k}} \circ FS_k (H_k) = H_k.
\]
\end{definition}
Then the quantized K\"ahler-Ricci soliton attached to $V_k$ are characterized as critical points of the quantization of the modified Ding functional ${\mathcal D}_{g_{V_k}}^{(k)}$. Moreover, by \cite[Proposition 4.7]{BN14}, we have
\[
{\mathcal D}_{g_{V_k}}^{(k)} (\exp(tW)^*H_0)=\frac{{\rm Fut}_{V_k, k} (W)}{k N_k}.
\]
\begin{proof}[Proof of the Theorem \ref{theo:1.2}]
This proof is mostly based on the original proof given by Berman-Nystr\"om. 
The reader should refer to \cite[Theorem 1.7]{BN14}.

The coercivity of ${\mathcal D}_{g_{V_{KS}}}$ implies that the equation
\[
{\mathcal D}_{g_{V_{KS}}} (FS_k(H)) \geq \delta J(FS_k(F^*H))-C
\]
holds for some $F \in {\rm Aut}_0 (X,V_{KS})$, where we note that two operations $FS_k$ and $F^*$ are commutative.
Then the LHS can be written as
\begin{eqnarray*}
{\mathcal D}_{g_{V_{KS}}} (FS_k(H)) &=& {\mathcal D}_{g_{V_{KS}}} (FS_k(F^*H)) \;\;\; (\text{because ${\rm Fut}_{V_{KS}} \equiv 0)$} \\
&=& J_{g_{V_{KS}}} (FS_k (F^*H)) +({\mathcal L} - {\mathcal L}_{\mu_0})(FS_k (F^*H)).
\end{eqnarray*}
On the other hand, since $g_{V_{KS}}$ is bounded, we obtain
\[
\delta J(FS_k(F^*H))-C \geq \delta' J_{g_{V_{KS}}}(FS_k(F^*H))-C
\]
for sufficiently small $\delta'>0$ depending only on $g_{V_{KS}}$.
Thus we obtain
\begin{equation}
J_{g_{V_{KS}}}(FS_k(F^*H))(1-\delta') + ({\mathcal L} - {\mathcal L}_{\mu_0})(FS_k (F^*H)) \geq -C. \label{eq:3.12}
\end{equation}
Now we use the following lemma, which compares the two functionals $J_{g_{V_{KS}}} \circ FS_k$ and $J_{g_{V_{KS}}}^{(k)}$:
\begin{lemma}[\cite{BN14}, Lemma 4.10] \label{lemm:3.4}
There exists a sequence $\delta_k \rightarrow 0$ of positive numbers such that
\[
J_{g_{V_{KS}}}(FS_k(H)) \leq (1+ \delta_k) J_{g_{V_{KS}}}^{(k)}(H) + \delta_k.
\]
\end{lemma}
Hence if we take $k$ sufficiently large so that $(1+\delta_k)(1-\delta') \leq 1-\frac{\delta'}{2}$ and $\delta_k(1-\delta') \leq C$ hold, we have 
\begin{equation}
J_{g_{V_{KS}}}(FS_k(F^*H))(1-\delta') \leq J_{g_{V_{KS}}}^{(k)} (F^*H) \left( 1-\frac{\delta'}{2} \right) + C. \label{eq:3.13}
\end{equation}
Thus we obtain
\begin{eqnarray*}
{\mathcal D}_{g_{V_{KS}}}^{(k)} (F^*H) &=& J_{g_{V_{KS}}}^{(k)} (F^*H) + ({\mathcal L} - {\mathcal L}_{\mu_0})(FS_k (F^*H)) \\
&\geq& \frac{\delta'}{2} J_{g_{V_{KS}}}^{(k)} (F^*H) -2C \;\;\;(\text{by \eqref{eq:3.12} and \eqref{eq:3.13}}) \\
&\geq& \frac{\delta''}{2} J^{(k)} (F^*H) -2C \;\;\;(\text{because $g_{V_{KS}}$ is bounded}).
\end{eqnarray*}
Now we consider the difference of the two modified Ding functionals:
\[
{\mathcal D}_{g_{V_{KS}}}^{(k)}-{\mathcal D}_{g_{V_k}}^{(k)}=-{\mathcal E}_{g_{V_{KS}}}^{(k)}+{\mathcal E}_{g_{V_k}}^{(k)},
\]
which has linear growth along geodesics explained above. 
On the other hand, the functional $J^{(k)}$ is an exhaustion function on ${\mathcal H}_k^{T_c}/ {\mathbb R}$ and has at least linear growth along geodesics (cf. \cite[Proposition 3]{Don09} or \cite[Lemma 7.6]{BBGZ13}). 
The following Lemma was inspired by these observations:
\begin{lemma} \label{lemm:3.5}
The inequality
\begin{equation}
-\epsilon_k J^{(k)} - \epsilon_k' \leq {\mathcal D}_{g_{V_{KS}}}^{(k)}-{\mathcal D}_{g_{V_k}}^{(k)} \leq \epsilon_k J^{(k)} + \epsilon_k' \label{eq:3.14}
\end{equation}
holds for some sequences of positive numbers $\epsilon_k \rightarrow 0$ and $\epsilon_k' \rightarrow 0$.
\end{lemma}
\begin{proof}
We set $\epsilon_k:= \sup_P |g_{V_{KS}} - g_{V_k}|+2^{-k}$ and define the functional ${\mathcal E}_{\epsilon_k + g_{V_{KS}} - g_{V_k}}^{(k)}:= \epsilon_k {\mathcal E}^{(k)} + {\mathcal E}_{g_{V_{KS}}}^{(k)} - {\mathcal E}_{g_{V_k}}^{(k)}$. Then we have $\epsilon_k \rightarrow 0$ since $V_k \rightarrow V_{KS}$ and $g_{V_k} \rightarrow g_{V_{KS}}$ uniformly on $P$. By scaling invariance of \eqref{eq:3.14}, we may assume that $H$ is normalized by ${\mathcal E}_{\epsilon_k + g_{V_{KS}} - g_{V_k}}^{(k)} (H)=0$. Now we consider a (non-trivial) geodesic $H_t$ starting at $H_0$ with eigenvalues $(\mu_i^{(k)})$. Then our normalization condition implies that $\mu_{max}:=\underset{i}{\max}(\mu_i^{(k)})$ is positive. Thus, computing in the similar way as in \cite[Proposition 3]{Don09}, we have
\[
(\epsilon_k J^{(k)}-{\mathcal D}_{g_{V_{KS}}}^{(k)}+{\mathcal D}_{g_{V_k}}^{(k)})(H_t) = \epsilon_k {\mathcal L}_{\mu_0} (FS_k(H_t)) \geq \epsilon_k \mu_{max} t + \text(const) \rightarrow \infty
\]
as $t \rightarrow \infty$. Hence the functional $\epsilon_k J^{(k)}-{\mathcal D}_{g_{V_{KS}}}^{(k)}+{\mathcal D}_{g_{V_k}}^{(k)}$ is coercive.
To get the second assertion $\epsilon_k' \rightarrow 0$, we use the $g$-analogue of calculation techniques developed in \cite[Section 7]{BBGZ13}. 
In what follows all $O(1)$ and $o(1)$ are meant to hold uniformly with respect to $H \in {\mathcal H}_k^{T_c}$ as $k \rightarrow \infty$.
We use the normalization
\[
{\mathcal L}_{\mu_0} (FS_k( H)) = 0
\]
so that
\begin{equation}
\sup_X (FS_k (H) - \phi_0) \leq O(1), \label{eq:3.15}
\end{equation}
and a reference point $\tilde{H}_0:= Hilb_{k, \mu_0, \epsilon_k + g_{V_{KS}} - g_{V_k}} (\phi_0)$. Let $H_t$ be a geodesic joining $\tilde{H}_0$ to $H:=H_1 \in {\mathcal H}_k^{T_c}$ and put
\[
v(H):= \left. \frac{\partial}{\partial t} \right|_{t=0} FS_k (H_t).
\]
Then we have the formula
\begin{equation}
{\mathcal E}_{\epsilon_k + g_{V_{KS}} - g_{V_k}}^{(k)} (H)= \int_X v(H) \beta_{k, \mu_0, \epsilon_k + g_{V_{KS}} - g_{V_k}} (\phi_0) + o(1), \label{eq:3.16}
\end{equation}
where the error term $o(1)$ in the RHS comes from the change of base points from $H_0$ to $\tilde{H}_0$. Then $v(H)$ is estimated as
\begin{eqnarray*}
v(H) &\leq& FS_k (H) - FS_k (\tilde{H}_0) \;\;\; (\text{by the convexity of $FS_k (H_t)$}) \\
&=& FS_k (H) - FS_k (Hilb_{k, \mu_0, \epsilon_k + g_{V_{KS}} - g_{V_k}}(\phi_0)) \\
&=& FS_k (H) - \phi_0 - \frac{1}{k} \log \left( \frac{1}{N_k} \rho_{k, \mu_0, \epsilon_k + g_{V_{KS}} - g_{V_k}} (\phi_0) \right) \\
&\leq& FS_k (H) - \phi_0 -\frac{1}{k} \log \left( \frac{1}{N_k} 2^{-k} \rho_{k, \mu_0} (\phi_0) \right) \\
&\hbox{}& (\text{by the definition of $\epsilon_k$}) \\
&\leq& O(1) \;\;\; (\text{by \eqref{eq:3.15} and Proposition \ref{prop:3.2}}).
\end{eqnarray*}
On the other hand, by the uniform covergence $g_{V_k} \rightarrow g_{V_{KS}}$ and Proposition \ref{prop:3.2}, the positive measure $\beta_{k, \mu_0, \epsilon_k + g_{V_{KS}} - g_{V_k}} (\phi_0)$ goes to $0$ uniformly as $k \rightarrow \infty$. Hence we obtain
\[
{\mathcal E}_{\epsilon_k + g_{V_{KS}} - g_{V_k}}^{(k)} (H) \leq \sup_X v(H) \int_X \beta_{k, \mu_0, \epsilon_k + g_{V_{KS}} - g_{V_k}} (\phi_0) + o(1) \leq \epsilon_k'
\]
for some positive number $\epsilon_k' \rightarrow 0$.
Therefore
\[
(\epsilon_k J^{(k)}-{\mathcal D}_{g_{V_{KS}}}^{(k)}+{\mathcal D}_{g_{V_k}}^{(k)}) (H) = - {\mathcal E}_{\epsilon_k + g_{V_{KS}} - g_{V_k}}^{(k)} (H) \geq - \epsilon_k'.
\]
One can prove another inequality in the similar way. 
\end{proof}
By Lemma \ref{lemm:3.5}, we have
\begin{eqnarray*}
{\mathcal D}_{g_{V_k}}^{(k)}(H) &=& {\mathcal D}_{g_{V_k}}^{(k)}(F^*H) \;\;\;(\text{Because ${\rm Fut}_{V_k, k} \equiv 0$}) \\
&\geq& {\mathcal D}_{g_{V_{KS}}}^{(k)} (F^*H) -\epsilon_k J^{(k)}  (F^*H)- \epsilon_k' \\
&\geq& \left( \frac{\delta''}{2} - \epsilon_k \right) J^{(k)} (F^*H) -2C-\epsilon_k'.
\end{eqnarray*}
Thus we have
\begin{equation}
{\mathcal D}_{g_{V_k}}^{(k)}(H) \geq \frac{\delta''}{3} \underset{F \in {\rm Aut}_0(X,V_{KS})}{\rm inf} J^{(k)} (F^*H) -3C \label{eq:3.17}
\end{equation}
for sufficiently large $k$.
Since $J^{(k)}$ is an exhaustion function on ${\mathcal H}_k^{T_c}/ {\mathbb R}$, we find that there exists a unique quantized K\"ahler-Ricci soliton $H_k$ at level $k$ up to the action of ${\rm Aut}_0(X,V_{KS})$ if $k$ is sufficiently large. We normalize $H_k$ so that the corresponding metric $\phi_k:=FS_k (H_k)$ minimizes $J$ on the corresponding ${\rm Aut}_0(X,V_{KS})$-orbit.
Then the minimizing property of $H_k$ implies ${\mathcal D}_{g_{V_k}}^{(k)} (H_k) \leq {\mathcal D}_{g_{V_k}}^{(k)} (Hilb_{k, \mu_0} (\phi))$ for all $\phi \in {\mathcal H}(X,-K_X)^{T_c}$. Thus letting $k \rightarrow \infty$, we obtain
\begin{equation}
{\mathcal D}_{g_{V_k}}^{(k)} (H_k) \leq {\mathcal D}_{g_{V_{KS}}} (\phi) + \gamma_k \label{eq:3.18}
\end{equation}
for all $\phi \in {\mathcal H}(X,-K_X)^{T_c}$,
where $\gamma_k= \gamma_k(\phi) \rightarrow 0$ is a sequence of constants depending on $\phi$.
On the other hand, we have
\begin{eqnarray*}
{\mathcal D}_{g_{V_{KS}}} (\phi_k) &\leq& {\mathcal D}_{g_{V_{KS}}}^{(k)}(H_k) +   \delta_k J^{(k)} (H_k) + \delta_k \;\;\; (\text{Lemma \ref{lemm:3.4} and $g_{V_{KS}}$ is bounded}) \\
&\leq& {\mathcal D}_{g_{V_k}}^{(k)}(H_k) +   \delta_k' J^{(k)} (H_k) + \delta_k' \;\;\; (\text{by Lemma \ref{lemm:3.5}}),
\end{eqnarray*}
where $J^{(k)} (H_k)$ is bounded from above by \eqref{eq:3.17} and \eqref{eq:3.18}. Thus we have
\[
\liminf_{k \rightarrow \infty} {\mathcal D}_{g_{V_{KS}}} (\phi_k) \leq {\mathcal D}_{g_{V_{KS}}} (\phi)
\]
for all $\phi \in {\mathcal H}(X,-K_X)^{T_c}$. Since the set ${\mathcal H}(X,-K_X)^{T_c}$ contains a K\"ahler-Ricci soliton with respect to $V_{KS}$, i.e., a minimizer of  ${\mathcal D}_{g_{V_{KS}}}$ on ${\mathcal E}^1 (X, -K_X)^{T_{KS}}$, we have
\[
\liminf_{k \rightarrow \infty} {\mathcal D}_{g_{V_{KS}}} (\phi_k) \leq \underset{\phi \in {\mathcal E}^1 (X, -K_X)^{T_{KS}}}{\inf} {\mathcal D}_{g_{V_{KS}}} (\phi).
\]
This yields that $\{\phi_k \}$ is a minimizing sequence of the functional ${\mathcal D}_{g_{V_{KS}}}$. Since $J^{(k)}(H_k)$ is bounded, $J(\phi_k)$ is also bounded by Lemma \ref{lemm:3.4}. Thus $\{ \phi_k \}$ is contained in a compact sublevel set of $J$, and there exists a subsequence which converges to some metric $\phi_{\infty} \in {\mathcal E}^1 (X, -K_X)^{T_{KS}}$. Since ${\mathcal D}_{g_{V_{KS}}}$ is lower semi-continuous  (cf: \cite[Lemma 6.4]{BBGZ13}, \cite[Proposition 2.15]{BN14}), we obtain
\[
{\mathcal D}_{g_{V_{KS}}}(\phi_{\infty}) \leq \liminf_{k \rightarrow \infty} {\mathcal D}_{g_{V_{KS}}} (\phi_k) \leq \underset{\phi \in {\mathcal E}^1 (X, -K_X)^{T_{KS}}}{\inf} {\mathcal D}_{g_{V_{KS}}} (\phi),
\]
hence $\phi_{\infty}$ is a K\"ahler-Ricci soliton with respect to $V_{KS}$, which is smooth by the regularity theorem \cite[Theorem 1.3]{BN14}. The metric $\phi_{\infty}$ may depend on a choice of a convergent subsequence. However, by our normalization of $\phi_k$, we know that $\phi_{\infty}$ minimizes $J$-functional on the space of K\"ahler-Ricci solitons with respect to $V_{KS}$, which can be identified with the space ${\rm Aut}_0 (X, V_{KS}) \phi_{\infty}/K$, where $K$ is the stabilizer of $\phi_{\infty}$ (cf. \cite[Theorem 3.6]{BN14}). Since $J$ is strictly convex on ${\rm Aut}_0 (X, V_{KS}) \phi_{\infty}/K$ with respect to the natural Riemannian structure (where geodesics are one parameter subgroups), such a minimizer is uniquely determined. Therefore, the metric $\phi_{\infty}$ is, in fact, independent of a choice of a subsequence, which yields that $\phi_k$ converges to a K\"ahler-Ricci soliton $\phi_{\infty}$ weakly. This completes the proof. 
\end{proof}

\end{document}